\documentclass[a4paper]{amsart}

\usepackage{graphicx}
\usepackage{color}
\usepackage{amssymb}
\usepackage[all]{xy}
\usepackage{amsmath}
\usepackage{epic}
\usepackage{subfigure}
\usepackage{epic}
\usepackage{eepic}
\usepackage{setspace}
\usepackage{booktabs}
\usepackage{longtable}
\usepackage{pstricks}
\usepackage{url}

\newcommand{\pD}{\mathcal{D}}
\newcommand{\D}{\mathcal{D}}
\newcommand{\sigmav}{\sigma^{\vee}}

\newcommand{\face}[2]{\text{face}\left(#1,#2\right)}

\newtheorem{thm}{Theorem}[section]
\newtheorem{lem}[thm]{Lemma}
\newtheorem{prop}[thm]{Proposition}
\newtheorem{cor}[thm]{Corollary}

\theoremstyle{definition}
\newtheorem{defn}[thm]{Definition}
\newtheorem{ex}[thm]{Example}
\newtheorem{rmk}[thm]{Remark}
\newtheorem{notation}[thm]{Notation}

\newcommand{\C}{\mathbb{C}}
\newcommand{\Z}{\mathbb{Z}}
\newcommand{\Q}{\mathbb{Q}}

\newcommand{\RR}{\mathbb{R}}

\newcommand{\QQ}{\mathbb{Q}}
\newcommand{\ZZ}{\mathbb{Z}}
\newcommand{\NN}{\mathbb{N}}
\newcommand{\PP}{\mathbb{P}}

\newcommand{\elements}[1]{\{#1\}}

\DeclareMathOperator{\linpart}{lin}

\DeclareMathOperator{\minvert}{min_{\text{vert}}}

\DeclareMathOperator{\relint}{relint}
\DeclareMathOperator{\tail}{tail}

\DeclareMathOperator{\SF}{SF}
\DeclareMathOperator{\tcadiv}{T-CaDiv}
\DeclareMathOperator{\tdiv}{T-Div}
\DeclareMathOperator{\ptdiv}{T-Princ}

\DeclareMathOperator{\Hom}{Hom}
\DeclareMathOperator{\divisor}{div}

\DeclareMathOperator{\wdiv}{Div}
\DeclareMathOperator{\pic}{Pic}

\DeclareMathOperator{\spec}{Spec}

\DeclareMathOperator{\ord}{ord}
\DeclareMathOperator{\supp}{supp}
\DeclareMathOperator{\xrays}{x-rays}
\DeclareMathOperator{\Pol}{Pol}
\DeclareMathOperator{\loc}{Loc}
\DeclareMathOperator{\pdv}{X}
\DeclareMathOperator{\Spec}{\mathbf{Spec}}

\DeclareMathOperator{\vol}{vol}
\DeclareMathOperator{\Graph}{\Gamma}
\DeclareMathOperator{\tcl}{T-Cl}
\DeclareMathOperator{\cl}{Cl}
\DeclareMathOperator{\dimension}{dim}
\DeclareMathOperator{\CaSF}{CaSF}
\DeclareMathOperator{\tpic}{T-Pic}
\DeclareMathOperator{\coeff}{coeff}

\renewcommand{\O}{{\mathcal{O}}}
\newcommand{\TV}{\mathbb{TV}}
\newcommand{\CO}{{\mathcal{O}}}

\newcommand{\fan}{\mathcal{S}}
\newcommand{\F}{\mathbb{F}}

\newcommand{\cotangfannullv}{%
 \psset{unit=0.6cm}
 \begin{pspicture}(-3.2,-3.2)(3.2,3.2)%
 \psset{linewidth=1pt}%
 \psline{-}(0,3)(0,1)(2,3)%
 \psline{-}(2,3)(0,1)(0,0)(3,0)%
 \psline{-}(3,0)(0,0)(0,-3)%
 \psline{-}(0,-3)(0,0)(-2,-2)%
 \psline{-}(-3, -3)(0,0)(0,1)(-3,1)%
 \psline{-}(-3, 1)(0,1)(0,3)%
 \rput[bl]{0}(0.5,2.4){$\D^1$}%
 \rput[bl]{0}(-2,2){$\D^2$}%
 \rput[bl]{0}(-2,-0.5){$\D^3$}%
 \rput[bl]{0}(-1.2,-2.2){$\D^4$}%
 \rput[bl]{0}(1.5,-1.5){$\D^5$}%
 \rput[bl]{0}(1.5,1.3){$\D^6$}%
\end{pspicture}}

\newcommand{\cotangfaninftyv}{% 
 \psset{unit=0.6cm}
 \begin{pspicture}(-3.2,-3.2)(3.2,3.2)%
  \psset{linewidth=1pt}%
  \psline{-}(0,3)(0,0)(3,3)%
  \psline{-}(2,2)(0,0)(0,0)(3,0)%
  \psline{-}(3,0)(0,0)(-1,-1)(-1,-3)%
  \psline{-}(-1,-3)(-1,-1)(-2,-2)%
  \psline{-}(-3, -3)(-1,-1)(-3,-1)%
  \psline{-}(-3, -1)(-1,-1)(0,0)(0,3)%
  \rput[bl]{0}(0.7,2){$\D^1$}%
  \rput[bl]{0}(-2,1){$\D^{2}$}%
  \rput[bl]{0}(-2.8,-2){$\D^{3}$}%
  \rput[bl]{0}(-2,-2.8){$\D^{4}$}%
  \rput[bl]{0}(1,-1.5){$\D^5$}%
  \rput[bl]{0}(2,0.7){$\D^6$}%
\end{pspicture}}

\newcommand{\cotangfanonev}{%
 \psset{unit=0.6cm}
 \begin{pspicture}(-3.2,-3.2)(3.2,3.2)%
 \psset{linewidth=1pt}%
 \psline{-}(0,3)(0,0)(1,0)(3,2)%
 \psline{-}(3,2)(1,0)(3,0)%
 \psline{-}(3,0)(1,0)(1,-3)%
 \psline{-}(1,-3)(1,0)(0,0)(-3,-3)%
 \psline{-}(-2, -2)(0,0)(-3,0)%
 \psline{-}(-3, 0)(0,0)(0,3)%
 \rput[bl]{0}(1,1.8){$\D^{1}$}%
 \rput[bl]{0}(-2,1.5){${\D^{2}}$}%
 \rput[bl]{0}(-2.4,-1.3){$\D^{3}$}%
 \rput[bl]{0}(-0.9,-2.2){$\D^{4}$}%
 \rput[bl]{0}(1.8,-1.8){$\D^{5}$}%
 \rput[bl]{0}(2.4,0.3){$\D^{6}$}%
\end{pspicture}}

\begin{document}

\title{Torus Invariant Divisors}

\author[L.~Petersen]{Lars Petersen}
\address{Institut f\"ur Mathematik und Informatik,
         Freie Universit\"at Berlin
	 Arnimallee 3,
	 14195 Berlin, Germany}
\email{petersen@math.fu-berlin.de}

\author[H.~S\"uss]{Hendrik S\"uss }
\address{Institut f\"ur Mathematik,
        LS Algebra und Geometrie,
        Brandenburgische Technische Universit\"at Cottbus,
        PF 10 13 44, 
        03013 Cottbus, Germany}
\email{suess@math.tu-cottbus.de}

\begin{abstract}
Using the language of \cite{MR2207875} and \cite{divfans}, we describe invariant divisors on normal varieties $X$ which admit an effective torus action of complexity one. In this picture, $X$ is given by a divisorial fan $\fan$ on a smooth projective curve $Y$. Cartier divisors on $X$  can be described by piecewise affine functions $h$ on the slices $\fan_P$ of $\fan$, whereas Weil divisors correspond to certain zero and one dimensional faces of it. Furthermore we provide descriptions of the divisor class group and the canonical divisor. Global sections of line bundles $\O(D_h)$ will be determined by a subset of a weight polytope associated to $h$, and global sections of specific line bundles on the underlying curve $Y$. \end{abstract}

\maketitle
%\tableofcontents

%%%%%%%%%%%%%%%%%%%%%%
\section{Introduction}
%%%%%%%%%%%%%%%%%%%%%%

Although toric geometry covers only a rather restricted class of varieties, it nevertheless provides a large amount of toy models and fruitful examples. In order to extend its results and techniques to a broader class of objects we suggest to consider varieties admitting an effective action of a lower dimensional torus, so-called $T$-varieties. 

In particular, one can consider $T$-varieties of complexity one, i.e.\ normal varieties $X$ of dimension $n$ which admit an effective action of the $(n-1)$-dimensional torus $T^{n-1}$. In flavour, this setting is still very close to the toric one, and there have already been several approaches, e.g. in \cite[Chapter IV, \S 1]{toroidal} via toroidal embeddings, in \cite{timashev_class} via the language of hypercones and hyperfans. The easiest class of examples is given by $\C^*$-surfaces, which have been studied in great detail, cf. e.g. \cite{MR0460342}, and \cite{MR2020670} and references therein.

However, this article will provide an insight into $T$-invariant divisors on $X$, i.e.\ Cartier and Weil divisors using the rather new language of polyhedral divisors. For comparison, the reader may consult \cite[Chapter II, \S\S 1,2]{toroidal} and \cite{timashevCDcomplOne}.

In section $2$ we recall the language of $T$-varieties from  \cite{MR2207875}, and \cite{divfans}. As we will specialize to actions of complexity one we display the essential features of this case. The building blocks for the description of a complexity-one $T$-variety $X$ are twofold: a smooth projective curve $Y$ as an algebro-geometric datum, and an $(n-1)$-dimensional divisorial fan $\fan$ on $Y$ as a combinatorial datum.

Section $3$ deals with invariant divisors. First, we consider Cartier divisors. Like in toric geometry they will be related to piecewise affine linear functions $h$ on the divisorial fan $\fan$. Then we will provide the description of Weil divisors which will follow easily from the orbit structure of $X$ lying over $Y$. We also include a formula for the divisor class group, and a representation of the canonical divisor. Moreover, we obtain a description of the global sections of a line bundle $\O(D_h)$ via a weight polytope $\Box_h$ associated to $h$, and global sections of specific line bundles on $Y$ induced by elements of $\Box_h$.

Section $4$ completes this paper by comparing parts of our results with those of \cite{MR2020670} in the case of affine $\C^*$-surfaces.

%%%%%%%%%%%%%%%%%%%%%%%%%%%%%%%%%%%%%%%%%%%%%
\section{T-Varieties} 
\label{sec:t-varieties}
%%%%%%%%%%%%%%%%%%%%%%%%%%%%%%%%%%%%%%%%%%%%%

We follow the notation of \cite{divfans}. First, let us recall some facts and notations from convex geometry. Let $N$ denote a lattice and $M:=\Hom(N,\Z)$ its dual.
The associated $\Q$-vector spaces $N \otimes \Q$ and $M \otimes \Q$ are denoted by $N_\Q$, and $M_\Q$, respectively. Let $\sigma \subset N_\Q$ be a pointed convex polyhedral cone. Consider a polyhedron $\Delta$ which can be written as a Minkowski sum $\Delta = \pi + \sigma$ of $\sigma$, and a compact polyhedron $\pi$. Then $\Delta$ is said to have $\sigma$ as its {\em tailcone}. This decomposition of $\Delta$ is only unique up to $\pi$.

With respect to Minkowski addition the polyhedra with tailcone $\sigma$ form a semi-group which we denote by $\Pol_{\sigma}^+(N)$. Note that $\sigma \in \Pol_{\sigma}^+(N)$ is the neutral element of this semi-group and that $\emptyset$  by definition is also an element of $\Pol_{\sigma}^+(N)$.

\begin{defn}
A \emph{polyhedral divisor} with tailcone $\sigma$ on a normal variety $Y$ is a formal finite sum
\[\D = \sum_Z \Delta_Z \otimes Z\,,\]
where $Z$ runs over all prime divisors on $Y$ and $\Delta_Z \in \Pol^+_{\sigma}(N)$. Here, finite means that only finitely many coefficients differ from the tailcone.
\end{defn}

For every element $u \in \sigma^\vee \cap M$ we can consider the evaluation of $\D$ via
\[\D(u):=\sum_Z \min_{v \in \Delta_Z} \langle u , v \rangle Z\,.\]
This yields an ordinary divisor on $\loc \D$, where
\[\loc \D := Y \setminus \left( \bigcup_{\Delta_Z = \emptyset} Z \right)\]
denotes the {\em locus} of $\D$.

\begin{defn}
A polyhedral divisor $\D$ is called {\em proper} if 
\begin{enumerate}
\item it is {\em Cartier}, i.e.\ $\D(u)$ is Cartier for every $u \in \sigma^\vee \cap M$,
\item it is {\em semiample}, i.e.\  $\D(u)$ is semiample for every $u \in \sigma^\vee \cap M$,
\item $\D$ is {\em big} outside the boundary, i.e.\ $\D(u)$ is big for every $u$ in the relative interior of $\sigma^\vee$.
\end{enumerate}
\end{defn}

From now on, we will only say polyhedral divisor instead of proper polyhedral divisor except for those cases in which we we want to distinguish between them explicitly. One can associate an $M$-graded $k$-algebra with such a polyhedral divisor, and consequently an affine scheme admitting a $T^N$-action:
\[\pdv(\D):= \spec \bigoplus_{u \in \sigma^\vee \cap M} \Gamma(\loc \D,\CO_{\loc \D}(\D(u)))\,.\]

More specifically, this construction gives us an affine normal variety of dimension $\dim Y + \dim N$ together with a $T^N$-action. Moreover, every normal affine variety with torus action can be obtained this way \cite{MR2207875}.

\begin{defn}
Let $\D=\sum_Z \Delta_Z \otimes Z$, and $\D'=\sum_Z \Delta'_Z \otimes Z$ be two polyhedral divisors on $Y$.
\begin{enumerate}
\item We write $\D' \subset \D$ if $\Delta'_Z \subset \Delta_Z$ holds for every prime divisor $Z$.
\item Let $\sigma : = \tail \D$ denote the tailcone of $\D$. For an element $u \in \sigmav \cap M$ we define $\face{\sigma}{u}$ to be the set of all $v \in \sigma$ such that $\langle u,v \rangle$ is minimal.
\item We define the {\em intersection} of polyhedral divisors by 
\[\D \cap \D' := \sum_Z (\Delta'_Z \cap \Delta_Z) \otimes Z.\]
\item We define the {\em degree} of a polyhedral divisor $\D$ on a curve $Y$ as 
\[\deg \D := \sum_Z \Delta_Z.\]
{\em Note: If $\D$ carries $\emptyset$-coefficients we obtain that $\deg \D = \emptyset$.}
\item For a (not necessarily closed) point $y \in Y$ we define the {\em fiber polyhedron} 
\[\D_y := \sum_{y \in Z} \Delta_Z.\]
{\em Note: We can recover $\Delta_Z$ this way since $\Delta_Z=\D_Z$.}
\item For an open subset $U \subset Y$ we set 
\[\D|_U := \D + \sum_{Z \cap U = \emptyset} \emptyset \otimes Z\,. \]
\end{enumerate}
\end{defn}

Now assume that $\D' \subset \D$ holds and $\D,\D'$ are proper. This implies
\[\bigoplus_{u \in \sigma^\vee \cap M} \Gamma\big(\loc \D',\CO_{\loc \D'}(\D'(u))\big)
\supset \bigoplus_{u \in \sigma^\vee \cap M} \Gamma\big(\loc \D,\CO_{\loc \D}(\D(u))\big),\]
and we get a dominant morphism $\pdv(\D') \rightarrow \pdv(\D)$.

\begin{defn}
If $\D' \subset \D$ holds for two proper polyhedral divisors and the corresponding map defines an open inclusion, then we say that $\D'$ is a {\em face} of $\D$, and we denote this by $\D' \prec \D$.
\end{defn}

\begin{defn}\ 
\begin{enumerate}
\item   A {\em divisorial fan} is a finite set $\fan$ of proper polyhedral divisors such that for $\D,\D' \in \fan$ we have $\D \succ \D' \cap \D \prec \D'$.
\item The polyhedral complex $\fan_y$ defined by the polyhedra $\D_y$ is called a {\em slice} of the divisorial fan $\fan$.
\item $\fan$ is called {\em complete} if all slices $\fan_y$ are complete subdivisions of $N_\QQ$.
\end{enumerate}
\end{defn}

The upper face relations guarantee that we can glue the affine varieties $\pdv(\D)$ via
\[\pdv(\D) \leftarrow \pdv(\D \cap \D') \rightarrow \pdv(\D').\]
By \cite[5.4.]{divfans} we know that the cocycle condition is fulfilled, so we obtain a variety which we denote by $\pdv(\fan)$. In the case of a complete this variety is also complete.

A divisorial fan $\fan$ corresponds to an open affine covering of $\pdv(\fan)$ given by $(X(\D))_{\D \in \fan}$. Observe that it is not unique, because we may switch to another invariant open affine covering of the same variety. We will do this occasionally by refining an existing divisorial fan.

Let us consider the affine case $X = \spec A$, where $A = \bigoplus_{u} \Gamma(\loc \D,\CO(\D(u)))$. 
We have $A_0 = \Gamma(\loc \D,\CO_{\loc \D})$, and thus get the following two proper and surjective maps to $Y_0 := \spec A_0$, the categorical quotient of $X$: 
\[q: X \rightarrow Y_0,\qquad \pi: \loc \D \rightarrow Y_0\,.\]

\begin{lem}\label{sec:lem-refine}
Let $\D$ be a polyhedral divisor on $Y$ and $\{U_i\}_{i \in I}$ an open affine covering of $Y_0$. 
Then $q^{-1}(U_i) \cong \pdv(\D|_{\pi^{-1}(U_i)})$. Moreover, we get a divisorial fan $\fan := \{\D|_{\pi^{-1}(U_i)}\}_{i \in I}$ such that $\pdv(\fan) \cong \pdv(\D)$.
\end{lem}

\begin{proof}
This is a direct consequence of \cite[3.3]{divfans}.
\end{proof}

\begin{rmk}
\label{sec:rmk-cod1}
We pay special attention to complexity-one torus actions. This means that the underlying variety $Y$ of the corresponding divisorial fan is a projective curve.

In this case, the locus of a polyhedral divisor is either affine or complete, and we get simple criteria for properness and the face relations:

\begin{itemize}
\item   $\D$ is a proper polyhedral divisor if $\deg \D$ is strictly contained in $\tail \D$ and for every $u \in \sigma^\vee$ with
\[\face{\tail \D}{u} \cap \deg \D \neq \emptyset\]
some multiple of $\D(u)$ is principal.
\item Given two polyhedral divisors $\D' \subset \D$ with $\D$ being proper, then $\D'$ is proper and a face of $\D$ if and only if $\Delta'_P$ is a face of $\Delta_P$ for every point $P \in Y$ and we have $\deg \D \cap \tail \D' = \deg \D'$.
\end{itemize}
\end{rmk}

Observe that $\pdv(\fan)$ is not determined by the prime divisor slices $\fan_D$ of $\fan$ in general. This can already be seen in the case of toric surfaces with restricted torus action (cf. Example \ref{sec:ex-hirzebruch}). Considering the Hirzebruch surface $\F_1$, we could blow up the point corresponding to the cone $\sigma_2$, thus inserting the ray $\RR_{\geq 0}(-1,0)$, and nevertheless obtain the same slices. So merely looking at the subdivisions, i.e.\ the slices $S_D$, does not give us all the necessary information. We really need to know which polyhedra in different slices belong to the same polyhedral divisor.

For a divisorial fan on a curve which consists only of polyhedral divisors with affine locus the situation is different. If we consider two such fans $\fan,\,\fan'$ having the same slices, Lemma \ref{sec:lem-refine} tells us that there exists a common refinement
\[\fan''=\{\D|_U \mid \D \in \fan, U \in \mathcal{U}\} = \{\D|_U \mid \D \in \fan', U \in \mathcal{U}\}\]
with $\mathcal{U}$ being a sufficiently fine affine covering  of $Y$. We then have
\[\pdv(\fan) \cong \pdv(\fan'') \cong \pdv(\fan').\]
For $Y$ a complete curve we may also have polyhedral divisors with locus $Y$. For reconstructing $\pdv(\fan)$ from the slices, we need to know which polyhedra belong to divisors with complete loci. 

In the forthcoming examples we will therefore label the maximal polyhedra in a subdivision by the polyhedral divisor they belong to. The locus of any polyhedral divisor $\D \in \fan$ can then be read off immediately.

\begin{rmk}
\label{sec:rem-toric}
\cite[sec. 5]{divfans} We get a very illuminating class of examples from toric geometry by restricting the torus action.

Let us consider a complete $n$-dimensional toric variety $X:=X(\Sigma)$. We restrict its torus action to that of a smaller torus $T \hookrightarrow T_X$ and construct a divisorial fan $\fan$ with $X(\fan)=X(\Sigma)$ in the following way. The embedding $T \hookrightarrow T_X$ corresponds to an exact sequence of lattices
\[0 \rightarrow N \stackrel{F}{\rightarrow} N_X \stackrel{P'}{\rightarrow} N' \rightarrow 0\,.\]
We may choose a splitting $N_X \cong N \oplus N'$ with projections
\[ P: N_X\rightarrow N,\quad P':N_X\rightarrow N'.\]
Define $Y:=X(\Sigma')$, where $\Sigma'$ is an arbitrary smooth projective fan $\Sigma'$ refining the images $P'(\delta)$ of all faces $\delta \in \Sigma$. Then every cone $\sigma \in \Sigma(n)$ gives rise to a polyhedral divisor $\D_{\sigma}$. For each ray $\rho' \in \Sigma'(1)$, let $n_{\rho'}$ denote its primitive generator. We then set
\[\Delta_{\rho'}(\sigma) = P(P'^{-1}(n_{\rho'})\cap \sigma)\,, \textnormal{ and} \;\, \D_\sigma = \sum_{\rho' \in \Sigma'(1)} \Delta_{\rho'}(\sigma) \otimes D_{\rho'}.\]
Finally $\{ \D_\sigma \}_{\sigma \in \Sigma(n)}$ is a divisorial fan. Observe that for certain polyhedral divisors $\D_\sigma$ and rays $\rho' \in \Sigma'(1)$ the intersection $P'^{-1}(n_{\rho'}) \cap \sigma$ may be empty. In this case we have that $\Delta_{\rho'}(\sigma) = \emptyset$. 
\end{rmk}

\begin{ex}
\label{sec:ex-hirzebruch}
We consider the $a$'th-Hirzebruch surface $\F_a$ as a $\C^*$-surface via the following maps of lattices
\[F = \left(\begin{smallmatrix} 1 \\ 0 \end{smallmatrix} \right)\,, \;\; P' = \left( \begin{smallmatrix} 0 & 1 \end{smallmatrix}\right)\,, \;\; P = \left(\begin{smallmatrix} 1 & 0 \end{smallmatrix}\right)\,.\] 

\begin{figure}[htbp]
\psset{unit=0.95cm}
\begin{pspicture}(0,-6)(12,0)
%\psgrid(0,0)(0,-8)(12,0)
\psframe[linecolor=white](0.5,-4.5)(3.5,-1.5)

\psline{->}(2,-3)(4,-3)
\psline{->}(2,-3)(0,-1)
\psline{<->}(2,-5)(2,-1)
\psline[linewidth=0.5pt, linestyle=dotted]{-}(0,-1.8)(4,-1.8)
\psline[linewidth=0.5pt, linestyle=dotted]{-}(0,-4.2)(4,-4.2)

\qdisk(0.5,-1.5){1pt}
\rput[bl]{0}(2.7,-2.15){$\sigma_0$}
\rput[bl]{0}(1.4,-2.15){$\sigma_1$}
\rput[bl]{0}(0.8,-3.3){$\sigma_2$}
\rput[bl]{0}(2.7,-4){$\sigma_3$}
\rput[tr]{0}(1.4,-1.2){\tiny{$(-1,a)$}}
\rput[bl]{0}(1.8,-5.5){$\F_a$}

\psline{<-|}(6,-3)(8,-3)
\psline{|->}(8,-3)(10,-3)
%\uput*[270](8,-3){$0$}
\rput[bl]{0}(10.7,-3){\textnormal{tailfan}}

\psline{<-|}(6,-1.8)(7.25,-1.8)
\psline{-|}(7.25,-1.8)(8,-1.8)
\psline{->}(8,-1.8)(10,-1.8)
\uput*[270](7.1,-1.8){${\tiny -\frac{1}{a}}$}
\rput[bl]{0}(11,-1.8){$\fan_0$}
\rput[bl]{0}(9,-1.6){$\D_{\sigma_0}$}
\rput[bl]{0}(6.5,-1.6){$\D_{\sigma_2}$}
\rput[bl]{0}(7.3,-1.6){$\D_{\sigma_1}$}

\psline{<-|}(6,-4.2)(8,-4.2)
\psline{|->}(8,-4.2)(10,-4.2)
\rput[bl]{0}(11,-4.2){$\fan_{\infty}$}
\rput[bl]{0}(9,-4){$\D_{\sigma_3}$}
\rput[bl]{0}(6.5,-4){$\D_{\sigma_2}$}
\rput[bl]{0}(8,-5.5){$\fan$}

\psline{|->}(5,-3)(5,-1)
\psline{|->}(5,-3)(5,-5)
%\qdisk(5,-1.8){1pt}
%\qdisk(5,-4.2){1pt}
%\qdisk(5,-3){1.5pt}
\rput[bl]{0}(4.5,-5.5){$Y=\PP^1$}

\end{pspicture}
\caption{Divisorial fan associated to $\F_a$.}
\end{figure}
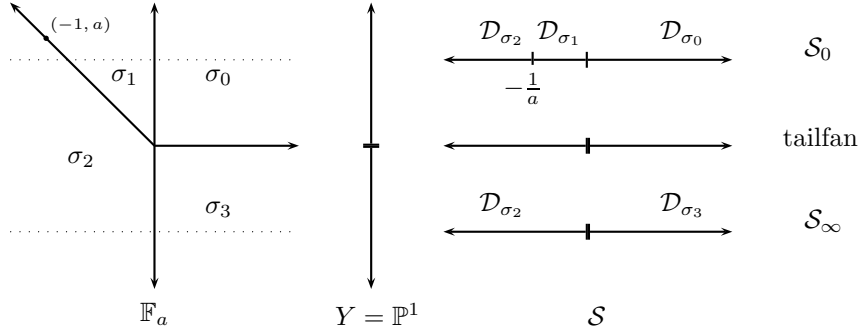

\end{ex}

%%%%%%%%%%%%%%%%%%%%%%%%%%%%%%%
\section{Invariant divisors}
\label{sec:invariant-divisors}
%%%%%%%%%%%%%%%%%%%%%%%%%%%%%%%

As we saw in the previous section the case of a torus action of complexity one can be handled quite comfortably. 
In particular, this is true for the concept of divisors on $T$-varieties. Therefore, we will restrict from now on to this case unless stated otherwise.

\begin{prop}
\label{sec:prop-complete-picard}
Let $\D$ be a polyhedral divisor with complete locus. Then $X:=\pdv(\D)$ has a trivial invariant Picard group $\tpic(X)$.
\end{prop}

\begin{proof}
Recall that the affine coordinate ring of $\pdv(\D)$ is
\[A:=\bigoplus_{u \in \sigma^\vee \cap M} \Gamma(Y,\CO(\D(u))).\]
We choose an element $v \in \relint(\tail \D)$  
and consider the homogeneous prime ideal
\[A_{>0} := \bigoplus_{\langle u , v \rangle > 0} A_u \,.\]
As the locus of $\D$ is complete we have $A_u = \Gamma(Y, \CO_Y) = k$ for $\langle u , v \rangle = 0$. This implies that every nonzero homogeneous element in $A \setminus A_{>0}$ is a unit. Therefore, every homogeneous ideal is a subset of $A_{>0}$, i.e.\ $V:=V(A_{>0})$ lies in the closure of every $T$-orbit on $X$. So the only invariant affine open subset that hits $V$ is $X$ itself. Hence, every invariant covering of $X$ contains $X$ itself implying that every invariant Cartier-divisor has to be principal.
\end{proof}

%%%%%%%%%%%%%%%%%%%%%%%%%%%%%%%%%%%%
\subsection{Cartier divisors}
\label{sec:construction-cartier-div}
%%%%%%%%%%%%%%%%%%%%%%%%%%%%%%%%%%%%
Our aim is to give a description of an invariant Cartier divisor $D$ on $X=X(\fan)$ in terms of a piecewise linear function on $\TV(\Sigma)$ and a divisor on the curve $Y$, where $\Sigma$ is the tailfan of $\fan$. The crucial input is the notion of a {\em divisorial support function}.

Let $\Sigma \subset N_\Q$ be a not necessarily complete polyhedral subdivision of $N_\Q$ consisting of tailed polyhedra. 

\begin{defn}
\label{def:support-function}
A continuous function $h: |\Sigma| \to \Q$ which is affine on every polyhedron $\Delta \in \Sigma$ is called a {\em $\Q$-support function}, or merely a {\em support function} if it has integer slope and integer translation, i.e.\ for $v \in |\Sigma|$ and $k \in \NN$ such that $kv$ is a lattice point we have $k h(v) \in \ZZ$.
The group of support functions on $\Sigma$ is denoted by $\SF(\Sigma)$.
\end{defn}

\begin{defn}
\label{def:linear-part}
Let $h$ be as above and $\Delta \in \Sigma$ a polyhedron with tailcone $\delta$. We define a linear function $h^\Delta_t$ on $\delta$ by setting $h^\Delta_t(v):= h(p+v)-h(p)$ for some $p \in \Delta$. As $h^\Delta_t$ is induced by $h$ we call it the {\em linear part} of $h|_\Delta$, or $\linpart h|_\Delta$ for short.
\end{defn}

Using Definitions \ref{def:support-function} and \ref{def:linear-part}, we can obviously associate a unique continuous piecewise linear function with an element $h \in \SF(\Sigma)$, say $h_t$. That is how we come to the crucial definition of this section. 

Let $\fan$ be a divisorial fan on a curve $Y$. For every $P \in Y$ we thus get a polyhedral subdivision $\fan_P$ consisting of polyhedral coefficients.

\begin{defn}
\label{def:fansy-support-function}
We define $\SF(\fan)$ to be the group of all collections
\[(h_P)_{P \in Y}\in \prod_{P \in Y}\SF(\fan_P) \quad \textnormal{such that}\]
\begin{enumerate}
\item all $h_P$ have the same linear part $h_t$, i.e.\ for polytopes $\Delta \in \fan_{P}$ and $\Delta' \in \fan_{P'}$ with the same tailcone $\delta$ we have that $\linpart h_P|_\Delta = \linpart h_{P'}|_{\Delta'} = h_t|_\delta$.
\item only for finitely many $P \in Y$ $h_P$ differs from $h_t$.
\end{enumerate}
We call $\SF(\fan)$ the group of {\em divisorial support functions} on $\fan$.
\end{defn}

\begin{notation}
We may restrict an element $h \in \SF(\fan)$ to a subfan or even to a polyhedral divisor $\D \in \fan$. The restriction will be denoted by $h|_\D$.
\end{notation}

\begin{defn}
A divisorial support function $h \in \SF(\fan)$ is called {\em principal} if $h(v) = \langle u, v \rangle + D$ with $u \in M$ and $D$ is a principal divisor on $Y$. Here, $D$ is to be considered as an element in $\SF(\fan)$ taking the constant value $\coeff_P(D)$ on every slice $\fan_P$.
\end{defn}

\begin{rmk}
Let us denote the function field of $Y$ by $K(Y)$. We then consider the graded ring $\bigoplus_{u \in M} K(Y)$, its multiplication being induced
by the one in $K(Y)$. Hence, we have a canonical inclusion of graded rings 
\[A:=\bigoplus_{u \in \sigma^\vee \cap M} \Gamma\big(\loc \D,\CO_{\loc \D}(\D(u))\big) \hookrightarrow \bigoplus_{u \in M} K(Y).\]
Moreover, the ring $\bigoplus_{u \in M} K(Y)$ is equal to the subring
\[K(X)^{\text{hom}} \subset K(X)=\text{Quot}(A)\]
which is generated by the semi-invariant functions on $X$, since $X$ is birationally equivalent to $T_M \times Y$ as a consequence of \cite[Thm.~3.1]{MR2207875}.

Thus, denoting the unit of $K(Y)$ in degree $u$ by $\chi^u \in K(X)^{\text{hom}}$ we obtain a unique representation $f \cdot \chi^u$ with $f \in K(Y)$ and $u \in M$ for every semi-invariant function on $X$.
\end{rmk}

% \begin{rmk}
% As we are going to relate divisorial support functions to $T$-invariant Cartier divisors we recall that $K(\widetilde{X})^T = K(Y)$ for $X = X(\D)$ affine. Then, $\pi: \widetilde{X} \to Y$ is a good quotient, cf. \cite[p.\,218]{divfans}. Here, $K(\widetilde{X})^T$ denotes the field of invariant rational functions on $\widetilde{X}$, whereas $K(Y)$ stands for the field of rational functions on $Y$. Furthermore, denote the ring generated by the semi-invariant rational functions on a $T$-variety $X = X(\fan)$ by $K(X)^{\text{hom}}$. As there is an equivariant birational morphism $r: \widetilde{X}(\fan) \to X(\fan)$ we can use an affine open subset $\widetilde{X}(\D)$ for $\D \in \fan$ to derive that $X$ is equivariantly birational to $U \times T^{n-1}$ for an open subset $U \subset Y$. The latter can be achieved by shrinking the support of $\D$ such that all polyhedral coefficients become trivial. Summing up all arguments we have that 
% \[K(X)^{\text{hom}} \cong \bigoplus_{u \in M} K(Y)\] 
% Moreover we have a canonical inclusion 
% \[
% \bigoplus_{u \in \sigma^\vee \cap M} \Gamma(\CO(\D(u))) \hookrightarrow K(X)^{\text{hom}},
% \]
% since $\Gamma(\CO(\D(u))) \subset K(Y)$.
% \end{rmk}

\begin{defn}
\label{def:casf}
A divisorial support function $h$ is called {\em Cartier} if for every $\D \in \fan$ with complete locus its restriction $h|_\D$ is principal. The corresponding group is denoted by $\CaSF(\fan)$.
\end{defn}

\begin{notation}
\label{not:tcadiv}
The group of $T$-invariant Cartier divisors on $X(\fan)$ is denoted by $\tcadiv(\fan)$.
\end{notation}

\begin{prop}
$\tcadiv(\fan)$ and $\CaSF(\fan)$ are isomorphic as abelian groups.
\end{prop}

\begin{proof}
Take an element $h=(h_P)_P \in \CaSF(\fan)$. Consider an element $\D \in \fan$ with a complete locus. Then we know by Proposition \ref{sec:prop-complete-picard} that every $T$-invariant Cartier divisor has to be principal on $\pdv(\D)$ and it is straightforward to see that principal support functions encode principal divisors. Now, for every $\D \in \fan$ with an affine locus there exist a weight $u(\D) \in M$ and a constant $a_P(\D) \in \ZZ$ such that $h_P|_\D = u(\D) + a_P(\D)$. For achieving this, we possibly have to pass to a refinement of $\fan$ as stated in Remark \ref{sec:rmk-cod1}. We may then cover $Y$ by open subsets $Y_i$ such that $\sum a_P(\D) P|_{Y_i}$ is principal on every $Y_i$. Let us assume that $\sum a_P(\D) P|_{Y_i} = - \divisor(f_\D^i)$ on $Y_i$. Then $f_\D^i \chi^{u(\D)} \in K(X)^{\textnormal{hom}}$, so it defines a principal divisor on $\pdv(\D|_{Y_i})$. It is then not hard to see that all of these principal divisors patch together to a Cartier divisor on  $\pdv(\fan)$.

On the other hand consider an element $D \in \tcadiv(\fan)$ given by an invariant open affine covering and local homogeneous generators. Intersecting with the open covering coming from $(X(\D))_{\D \in \fan}$ yields an open invariant cover for every affine $X(\D)$ which itself is induced by a cover of $\loc \D$. Let us denote this induced covering of $X(\D)$ by $(U_i)_{i \in I}$ with local generators of $D$ denoted by $f_i \in K(X)^*$. Depending on whether $\loc \D$ is affine or complete, we have
\[D|_{X(\D)} = (U_i,f_i\chi^{u_i})_i \quad \textnormal{or} \quad  D|_{X(\D)} = (X(\D),f\chi^u)\,.\]
Recall that for an element $\D \in \fan$ with complete locus there is no choice since we only have the trivial covering. To construct $h_t$ we choose for every maximal cone $\sigma \in \tail \fan$ an element $\D$ having exactly this tailcone and assign the weight $-u_i$ to it. Observe that this choice is independent of $i$. Doing this for all maximal cones $\sigma$ finally gives us $h_t$. Now, if $P \in Y$ is in $\supp \,(\divisor f_i)$ for some $f_i$ we set $a_P(\D) := -\coeff_P(\divisor f_i)$. This again is independent of the choice of $i$ and all these affine functions these functions fit together to an element $h \in \CaSF(\fan)$.
\end{proof}

\begin{notation}
The $T$-invariant Cartier divisor induced by an element $h \in \CaSF(\fan)$ is denoted by $D_h$. We will often switch from one notation to the other and consider both expressions to mean the same thing.
\end{notation}

As an immediate consequence we have this

\begin{cor}
\label{sec:tpic}
The $T$-invariant Picard group of $X(\fan)$ is given by
\[\tpic X \cong \frac{\CaSF(\fan)}{\langle \langle u,\cdot\rangle + D \, | \, u \in M, D \sim 0\rangle}\,.\]
\end{cor}

%%%%%%%%%%%%%%%%%%%%%%%%%%
\subsection{Weil divisors}
%%%%%%%%%%%%%%%%%%%%%%%%%%
We split this section into two parts as since we will give some results concerning torus actions of arbitrary complexity. We will deal with this case in the first part. The second part can be regarded as a specialization of the first one but we also provide further results not yet obtained in the general case.

\subsubsection*{Torus Actions of Arbitrary Codimension}

We would like to describe $T$-invariant prime divisors. As $X(\fan)$ is patched together by affine charts $X(\D)$ we can restrict to the affine case. Set $n:=\dimension(T) = \dimension(X(\fan)) -k$, where $k$ is the dimension of the base variety $Y$. We can assume the latter to be smooth and projective. 

In general there are two types of $T$-invariant prime divisors:

\begin{enumerate}
\item families of $n$-dimensional orbit closures over prime divisors in $X(\fan)$.\label{item:dimT}
\item families of $n-1$-dimensional orbit closures over $X(\fan)$.\label{item:dimT-1}
\end{enumerate}

\begin{prop}
\label{sec:prop-prime-divs}
Let $\D$ be a polyhedral divisor with tailcone $\sigma$ on an arbitrary normal variety $Y$, then there are one-to-one correspondences 
\begin{enumerate}
\item between prime divisor of {\em type~\ref{item:dimT}} and vertices $v \in \Delta_Z$ with $Z$ being a prime divisor on $Y$, such that $\CO(\D(u))|_Z$ is big, for $u \in ((\Delta_Z - v)^\vee)^\circ$.
\item between prime divisors of {\em type~\ref{item:dimT-1}} and rays $\rho$ of $\tail \D$ with $\D(u)$ big for $u \in (\rho^\perp \cap \sigma^\vee)^\circ$.
\end{enumerate}
\end{prop}

\begin{proof}
Consider $\widetilde{X}:=\Spec_{\loc \D} \bigoplus_u \CO(\D(u))$. We have $\pdv(\D) = \spec \Gamma(\widetilde{X}, \CO_{\widetilde{X}})$ and get equivariant maps
\[\xymatrix{\loc \D & {\widetilde{X}} \ar[l]_{\quad \; \pi} \ar[r]^{r} & X}.\]
From \cite[3.1]{MR2207875} we know that $\pi$ is a good quotient map, and $r$ is a birational morphism. In \cite[7.11]{MR2207875} the orbit structure of the fibers of $\pi$ is described. Thus, we know that $l$-dimensional faces $F$ of $\D_y$ correspond to $T$-invariant closed subvarieties of codimension $l$ in $\pi_y:=\pi^{-1}(y)$. While stated only for closed points one checks that the proof in fact works equally well for generic points $\xi$.

Furthermore we have to consider those subvarieties that get contracted by $r$. By \cite[10.1]{MR2207875} we have that
\[\dim Z - \dim r(Z) = \dim \pi(Z) - \dim \vartheta_u(\pi(Z))\]
for any invariant subvariety $Z \subset \widetilde{X}$. 

Hence, the bigness condition is equivalent to the fact that the image under $r$ of the corresponding prime divisor in $\widetilde{X}$ is again of codimension $1$.
\end{proof}

\begin{prop}
\label{sec:prop-principal}
We consider a polyhedral divisor on an arbitrary normal variety $Y$. Let $f \cdot \chi^u \in K(X)^{\textnormal{hom}}$. Then the corresponding principal divisor is given by 
\[\sum_{\rho} \langle u,n_\rho \rangle D_\rho  + \sum_{(Z,v)} \mu(v) (\langle u, v\rangle + \ord_Z f) D_{(Z,v)},\]
where $\mu(v)$ is the smallest integer $k \geq 1$ such that $k \cdot v$ is a lattice point, this lattice point is a multiple of the primitive lattice vector: $\mu(v)v=\varepsilon(v) n_v$.
\end{prop}

\begin{proof}
This is a local statement, so we will pass to a sufficiently small invariant open affine set which meets a particular prime divisor. If we translate this into our combinatorial language and consider a prime divisor corresponding to $(Z,v)$ or $\rho$ then we have to choose a polyhedral divisor $\D' \prec \D \in \fan$ such that $v$ is also a vertex of $\D'_Z$ or $\rho$ is a ray in $\tail \D'$, respectively.

So we restrict to following two affine cases:
\begin{enumerate}
\item  $\D$ is a polyhedral divisor with tailcone $\sigma=0$ and a single point $\Delta_Z = v \in N_\QQ$ as the only nontrivial coefficient. Moreover, $Y$ is affine and factorial. In particular, $Z$ is a prime divisor with (local) parameter $t_Z$.
\item $\D$ is the trivial polyhedral divisor with one dimensional tailcone $\rho$ over an affine locus $Y$.
\end{enumerate}  
 
In the first case we may choose a $\ZZ$-Basis $e_1,\ldots, e_m$ of $N$ with 
$e_1 = n_v$ and consider the dual basis $e^*_1,\ldots, e^*_m$,  
By definition $\varepsilon(v)$ and $\mu(v)$ are coprime, so we will find $a,b \in \ZZ$ such that
$a \mu(v) + b \varepsilon(v) = 1$. In this situation $y:=t_Z^a\chi^{be^*_1}$ is irreducible in
\[\Gamma(\CO_X) = \Gamma(\CO_Y)[y, t_Z^{\pm \varepsilon(v)}\chi^{\mp \mu (v) e^*_1},\chi^{\pm e^*_2}, \ldots ,\chi^{\pm e^*_m}],\] 
and defines the prime divisor $D_{(Z,v)}$.
We consider an element $t_Z^\alpha \chi^{u}$ with $u=\sum_i \lambda_i e^*_i$. 
The $y$-order of $t_Z^\alpha \chi^{u}$ is $\varepsilon(v)\lambda_1 + \mu(v)\alpha = \mu(v)(\langle u, v \rangle+\alpha)$, since
\[t_Z^\alpha \chi^{u} = y^{\varepsilon(v)\lambda_1 + \mu(v)\alpha} (t_Z^{- \varepsilon(v)}\chi^{\mu (v) e^*_1})^{\lambda_1 a + b \alpha}\chi^{\lambda_2 e^*_2}\cdots\chi^{\lambda_m e^*_m},\] 
and $t_Z^{- \varepsilon(v)}\chi^{\mu (v) e^*_1}$ is a unit.

In the second case we choose a $\ZZ$-basis  $e_1,\ldots, e_m$ of $N$ with $e_1 = n_\rho$. Once again we consider the dual basis $e^*_1,\ldots, e^*_m$. In this situation
\[\Gamma(\CO_X) = \Gamma(\CO_Y)[\chi^{e^*_1},\chi^{\pm e^*_2}, \ldots ,\chi^{\pm e^*_m}].\]
Now $(\chi^{e^*_1})$  defines the prime divisor $\rho$ on $X$. For a principal divisor $f \cdot \chi^u$, the $\chi^{e^*_1}$-order equals the $e^*_1$-component of $u$, i.e.\ $\langle u, n_\rho \rangle$.
\end{proof}

\noindent
Our next goal is to describe the divisor class group of $X(\fan)$. Denote by $\tdiv(X(\fan))$ the $T$-invariant divisors, and by $\ptdiv(X(\fan))$ the $T$-invariant principal divisors on $X(\fan)$. Then we have that 
\[\cl\big(X(\fan)\big) \cong \tcl(X(\fan)) := \frac{\tdiv(X(\fan))}{\ptdiv(X(\fan))}\,.\]

\begin{cor}
\label{cor:divclass}
The divisor class group of $X(\fan)$ is given by
\[\cl\big(X(\fan)\big) = \frac {\bigoplus_\rho \ZZ \cdot D_\rho \oplus \bigoplus_{D_{(Z,v)}} \ZZ \cdot D_{(Z,v)}}{\langle \sum u(n_\rho) D_\rho + \sum_{D_{(Z,v)}} \mu(v) (\langle u, v\rangle + a_Z) D_{(Z,v)} \rangle}.\]
Here $u$ runs over all elements of $M$ and $\sum_Z a_Z Z$ over all principal divisors on $Y$. Thus, it is isomorphic to
\[\cl(Y) \oplus \bigoplus_\rho \ZZ D_\rho \oplus \bigoplus_{D_{(Z,v)}} \ZZ D_{(Z,v)}\]
modulo the relations

\begin{eqnarray*}
[Z] &=& \sum_{v\in \fan_Z} \mu(v)D_{(Z,v)}\,,\\
 0  &=& \sum_{\rho}  \langle u,\rho \rangle D_\rho  + \sum_{D_{(Z,v)}} \mu(v) \langle u,v \rangle D_{(Z,v)}\,.
\end{eqnarray*}
\end{cor}

\begin{rmk}
We can also describe the ideals of prime divisors in terms of polyhedral divisors:
\begin{enumerate}
\item For prime divisors of type~\ref{item:dimT} corresponding to $(Z,v)$ the ideal is given by
\[I_{D_{(Z,v)}} = \bigoplus_{u \in \sigma^\vee \cap M} \Gamma(Y,\CO(\D(u))) \cap \{f \in K(Y) \mid \ord_Z(f) > - \langle u,v \rangle\}.\]
\item For prime divisors of type~\ref{item:dimT-1} the corresponding ideal is generated by all degrees $u$ which are not orthogonal to $\rho$:
\[I_{D_\rho} = \bigoplus_{u \in \sigma^\vee  \setminus \rho^\perp \cap M} \Gamma(Y,\CO(\D(u))).\]
\end{enumerate}
\end{rmk}

\subsubsection*{Torus Actions of Complexity One}
Stepping back to the complexity one case, Proposition \ref{sec:prop-prime-divs} is equivalent to 

\begin{cor}
  \label{sec:cor-prime-divs}
  Let $\D$ be a polyhedral divisor on a curve $Y$. Then there are one-to-one correspondences 
  \begin{enumerate}
  \item between prime divisors of {\em type~\ref{item:dimT}} and pairs $(P,v)$ with $P$ being point on $Y$ 
    and $v$ a vertex of $\Delta_P$,
  \item between prime divisors of {\em type~\ref{item:dimT-1}} and rays $\rho$ of $\sigma$ with $\deg \D \cap \rho = \emptyset$.
  \end{enumerate}
\end{cor}

\begin{defn}
Let $\D \in \fan$ be a polyhedral divisor with tailcone $\sigma$. A ray $\rho \prec \sigma$ with $\deg \D \cap \rho = \emptyset$ is called an {\em extremal ray}. The set of extremal rays is denoted by $\xrays(\D)$ or $\xrays(\fan)$, respectively.
\end{defn}

The combination of Proposition \ref{sec:prop-principal} and the description of $T$-Cartier divisors yield  

\begin{cor}
\label{sec:cor-cartier2weil}
Let $h = \sum_P h_P$ be a Cartier divisor on $\D$. Then the corresponding Weil divisor is given by 
\[-\sum_{\rho} h_t (n_\rho) D_\rho - \sum_{(P,v)} \mu(v) h_P(v) D_{(P,v)}.\]
\end{cor}

\begin{cor} 
\label{sec:cor-picard-rank}
Assume $X=X(\fan)$ to be a complete $\QQ$-factorial variety of dimension $n+1$. Denote by $\fan^{(0)}_P$ the set of vertices in $S_P$. Then the Picard number of $X$ is given by
\[\rho_X= 1 + \# \xrays(\fan) + \sum_{P \in Y} (\# \fan_P^{(0)} - 1) - n.\]
\end{cor}

\begin{thm}
\label{sec:thm-canonical-divisor}
For the canonical class of $X=X(\fan)$ we have that
\[K_X = \sum_{(P,v)} \left(\mu(v) K_Y(P) + \mu(v) - 1 \right)\cdot D_{(P,v)} - \sum_{\rho}D_\rho \,.\]
\end{thm}

\begin{proof}
Let $\omega_Y \in \Omega^1(Y)$ a (rational) differential form. Then $K_Y$ is given by $K_Y=\divisor \omega_Y$.
For a given $P \in Y$ we have a representation $\omega_Y= f_P d t_P$, where $f_P \in K(Y)$ and $t_P$ a local parameter of $P$

We define a differential form $\omega_X$ by 
\[\omega_X = \omega_Y \wedge \frac{d\chi^{e^*_1}}{\chi^{e^*_1}} \wedge \ldots \wedge \frac{d\chi^{e^*_n}}{\chi^{e^*_n}},\]
with $e^*_1,\ldots e^*_n$ being a $\ZZ$-basis of $M$.

For a prime divisor $(P,v)$ we may choose a $\ZZ$-Basis $e_1,\ldots, e_n$ of $N$ with  $e_1 = n_v$. Consider the dual basis $e^*_1,\ldots, e^*_n$. As $\mu$ and $\varepsilon(v)$ are coprime we may choose $a,b \in \ZZ$ with $a \mu(v) + b \varepsilon(v)=1$. Hence, $t_P^a \chi^{be^*_1}$ is a local parameter associated to $(P,v)$. It is then easy to see that we have the following local representation 
\[\omega_X= \frac{f_P}{a t_P^{a-1} \chi^{be^*_1}} d (t_P^a \chi^{be^*_1}) \wedge \frac{d\chi^{e^*_1}}{\chi^{e^*_1}} \wedge \ldots \wedge \frac{d\chi^{e^*_n}}{\chi^{e^*_n}}.\]
Then Corollary \ref{sec:cor-cartier2weil} implies $\ord_{D_{(P,v)}}(\frac{f_P}{a t_P^{a-1} \chi^{be^*_1}}) = (\ord_P(f_P)+1)\mu(v) - 1$.

For a prime divisor $D_\rho$ of $X$ we choose a $\ZZ$-basis  $e_1,\ldots, e_n$ of $N$ with $e_1 = n_\rho$. Again we consider the dual basis $e^*_1,\ldots, e^*_n$. Then $\chi^{e^*_1}$ is a local parameter for $D_\rho$ and we have a local representation
\[\omega_X = \omega_Y \wedge \frac{d\chi^{e^*_1}}{\chi^{e^*_1}} \wedge \ldots \wedge \frac{d\chi^{e^*_n}}{\chi^{e^*_n}},\]

We immediately find that $\ord_{D_\rho}(\omega_Y)=0$ and $\ord_{D_\rho}(\frac{d\chi^{e^*_1}}{\chi^{e^*_1}} \wedge \ldots \wedge \frac{d\chi^{e^*_n}}{\chi^{e^*_n}})=-1$.
\end{proof}

%\begin{rmk}
%We can also describe the ideals of prime divisors in terms of polyhedral divisors:
%\begin{enumerate}
%\item For prime divisors of type~\ref{item:dimT} corresponding to a vertex $(P,v)$ the ideal is given by
%\[I_{(P,v)} = \bigoplus_{u \in \sigma^\vee \cap M} \Gamma(Y,\CO(\D(u))) \cap \{f \in K(Y) \mid \ord_P(f) > \langle v , u \rangle\}.\]
%\item For prime divisors of type~\ref{item:dimT-1} the corresponding ideal is generated by all degrees $u$ which are not orthogonal to $\rho$:
%\[I_{\rho} = \bigoplus_{u \in \sigma^\vee  \setminus \rho^\perp \cap M} \Gamma(Y,\CO(\D(u))).\]
%\end{enumerate}
%\end{rmk}

%%%%%%%%%%%%%%%%%%%%%%%%%%%%%
\subsection{Global sections}
\label{sec:global-sections} 
%%%%%%%%%%%%%%%%%%%%%%%%%%%%% 
For an invariant Cartier-Divisor $D_h$ on $X$ we may consider the $M$-graded module of global sections 
\[L(h)= \bigoplus_{u \in M} L(h)_u = \bigoplus_{u \in M} L(D_h)_u := \Gamma(X,\CO(D_h)).\]
The weight set of $h$ is defined as $\{u \in M \mid L(h)_u \neq 0\}$. 
For a Cartier divisor $D_h$ given by $h \in \CaSF(\fan)$ we will bound its weight set by a polyhedron and describe 
the graded module structure of $L(h)$.

\begin{defn}
Given a support function $h = (h_P)_P$ with linear part $h_t$ its {\em associated polyhedron} is given by
\[\Box_h := \Box_{h_t}:=\left\{u \in M_\QQ \mid  \langle u, v \rangle  \geq  h_t(v) \; \textnormal{ for all } v \in |\tail \fan| \right\}\,.\]
We furthermore define the {\em dual function} $h^*:\Box_h \rightarrow \wdiv_\QQ Y$ to $h$ by
\[h^*(u):=\sum_P h_P^*(u)P := \sum_P \minvert(u-h_P)P,\]
where $\minvert(u - h_P)$ denotes the minimal value of $u - h_P$ on the vertices of $\fan_P$.
\end{defn}

\begin{prop}
\label{prop:global_sections}
Let $h \in \tcadiv(\fan)$  be a Cartier divisor with linear part $h_t$. Then
\begin{enumerate}
\item The weight set of $L(h)$ is a subset of $\Box_{h}$.
\item For $u \in \Box_{h}$ we have
\[L(h)_u = \Gamma(Y,\CO(h^*(u))).\]
\end{enumerate}
\end{prop}

\begin{proof}
By definition of $\CO(D_h)$ we have
\[\Gamma(X,\CO(h))^{T} = \{ f\chi^u   \mid  \divisor(f\chi^u) - \sum_{\rho} h_t (n_\rho) D_\rho - \sum_{(P,v)} \mu(v) h_P(v) D_{(P,v)} \geq 0\}.\]
But $\divisor(f\chi^u) = \sum_{\rho} \langle u, n_\rho \rangle D_\rho + \sum_{(P,v)} \mu(v) (\langle u, v \rangle + \ord_P(f)) D_{(P,v)}$, so for $f\chi^u \in L(h)$ we get the following bounds:
\begin{enumerate}
\item $\langle u, n_\rho \rangle \geq h_t(n_\rho) \; \forall \,\rho$
\item $\ord_P(f) + \langle u, v \rangle \geq  h_P(v) \; \forall \,(P,v)$
\end{enumerate}
The first implies that $u \in \Box_{h_t}$, whereas the second yields $\ord_P(f) + (u - h_P)(v) \geq 0$ for all $(P,v)$.
\end{proof}

\begin{ex}
\label{ex:hirzebruch}
Let us consider the Hirzebruch surface $\F_2$ together with the line bundle $L=\O(D)$ which is given through the following generators on the maximal cones $\sigma_i$
\[u_{\sigma_0} = [0 \; 0]\,,\quad u_{\sigma_1} = [1 \; 0]\,,\quad u_{\sigma_2} = [3 \; 1]\,,\quad u_{\sigma_3} = [0 \; 1]\,.\]
It is very ample and defines an embedding into $\PP^5$. One can describe the embedding by a polytope $\Delta \subset M_{\Q} = \Q^2$ which is the convex hull of the $u_{\sigma_i}$ and has six lattice points which are a basis of $\Gamma(\mathbb{F}_2,\O(D))$.

Considering the piecewise linear function $\phi_L$ given by $\{u_{\sigma_i}\;|\; 0 \leq i \leq 3 \}$, one obtains the graph of $h_P$ over $\fan_P$  by evaluating $\phi_L$ along the dotted slices corresponding to $P$. As usual $h_t$ denotes the piecewise linear function on the tailfan. In the following figure, the number over each cone denotes the slope of the corresponding restriction of $h_P$.

\begin{figure}[htbp]
\label{fig:hirzebruch}
\psset{unit=0.95cm}
\begin{pspicture}(0,-6)(13,0)
%\psgrid(0,0)(0,-8)(12,0)
\psline{<-|}(1,-3)(3,-3)
\psline{->}(3,-3)(5,-3)
%\uput*[270](3,-3){$0$}
\rput[bl]{0}(12,-3){$h_t$}

\psline{<-|}(1,-1.5)(2,-1.5)
\psline{-|}(2,-1.5)(3,-1.5)
\psline{|->}(3,-1.5)(5,-1.5)
\uput*[270](2,-1.5){{\tiny $-\frac{1}{2}$}}
\rput[bl]{0}(12,-1.5){$h_0$}
\rput[bl]{0}(4,-1.3){$\D_{\sigma_0}$}
\rput[bl]{0}(1.2,-1.3){$\D_{\sigma_2}$}
\rput[bl]{0}(2.2,-1.3){$\D_{\sigma_1}$}

\psline{<-|}(1,-4.5)(3,-4.5)
\psline{->}(3,-4.5)(5,-4.5)
\rput[bl]{0}(12,-4.5){$h_{\infty}$}
\rput[bl]{0}(4,-4.3){$\D_{\sigma_3}$}
\rput[bl]{0}(1.5,-4.3){$\D_{\sigma_2}$}

%%%%%%%%%%%%%%%%%%%%%%%%%%%%%%%%%%%%%%%%%%%%%%%%%
%%%%%%%%%%%%%%%%%%%%%%%%%%%%%%%%%%%%%%%%%%%%%%%%%

\psline[linestyle=dotted]{<-|}(7,-3)(9,-3)
\psline{->}(9,-3)(11,-3)
\psline(7,-4)(9,-3)
\psline(9,-3)(11,-3)
\uput*[90](10,-3){$0$}
\uput*[90](8,-3){$3$}

\psline[linestyle=dotted]{<-|}(7,-1.5)(8,-1.5)
\psline[linestyle=dotted]{-|}(8,-1.5)(9,-1.5)
\psline[linestyle=dotted]{->}(9,-1.5)(11,-1.5)
\psline(7,-2.16)(8,-1.66)
\psline(8,-1.66)(9,-1.5)
\psline(9,-1.5)(11,-1.5)
\uput*[90](10,-1.5){$0$}
\uput*[90](8.5,-1.5){$1$}
\uput*[90](7.5,-1.5){$3$}

\psline[linestyle=dotted]{<-|}(7,-4.5)(9,-4.5)
\psline[linestyle=dotted]{|->}(9,-4.5)(11,-4.5)
\psline(7,-6)(9,-5)
\psline(9,-5)(11,-5)
\uput*[90](10,-4.5){$0$}
\uput*[90](8,-4.5){$3$}

\rput[bl]{0}(3,-6){$\fan$}
\rput[bl]{0}(9,-6){$h$}

\end{pspicture}
\caption{The graphs of $h_0$, $h_t$ and $h_\infty$ over the corresponding slice $\fan_P$.}
\end{figure}
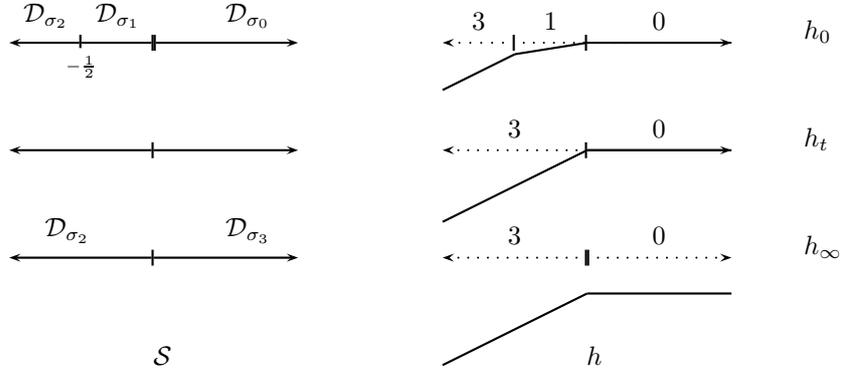

By Proposition \ref{prop:global_sections} we have $\Box_h = \{ u\in \Z\, | \; 3 \geq u \geq 0 \}$, and
\[\begin{array}{ll} L(h)_0 = \Gamma \big(\PP^1,\O(\{\infty\})\big)\,, & L(h)_1 = \Gamma \big(\PP^1,\O(\{\infty\})\big)\,, \\ L(h)_2 = \Gamma \big(\PP^1,\O(\{\infty\}- 1/2\{0\})\big)\,, & L(h)_3 = \Gamma \big(\PP^1,\O(\{\infty\}-\{0\})\big)\,. \end {array} \]
Altogether they sum up to a six dimensional vector space. We complete the example by a figure of $h^*$, see Figure \ref{fig:hstarHirzebruch}.

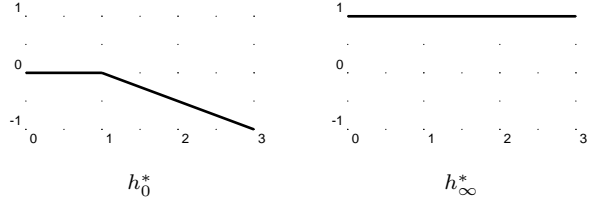
\begin{figure}[h]
\centering
\subfigure[$h^*_0$]{
 \psset{yunit=0.75cm}
 \begin{pspicture}(0,-1.5)(3,1.5)%
  \psgrid[gridwidth=0.5pt,griddots=2,subgriddiv=0,gridlabels=5pt](0,-1.5)(3,1.5)
  \psline[linewidth=1pt]{-}(0,0)(1,0)
  \psline[linewidth=1pt]{-}(1,0)(3,-1)
  %\psline[linewidth=1pt,linecolor=gray]{-}(1.05,0)(3,0)
 \end{pspicture}}\hspace*{1cm}
\subfigure[$h^*_{\infty}$]{
 \psset{yunit=0.75cm}
 \begin{pspicture}(0,-1.5)(3,1.5)%
  \psgrid[gridwidth=0.5pt,griddots=2,subgriddiv=0,gridlabels=5pt](0,-1.5)(3,1.5)
  \psline[linewidth=1pt]{-}(0,1)(3,1)
  %\psline[linewidth=1pt,linecolor=gray]{-}(0,0)(3,0)
 \end{pspicture}}
\caption{The graphs of $h^*_0$ and $h^*_\infty$ over $\Box_h$, cf.\ Example \ref{ex:hirzebruch}.}
\label{fig:hstarHirzebruch}
\end{figure}

\end{ex}

\begin{ex} 
\label{ex:cotan}
As another example we consider $X =\PP(\Omega_{\PP^2})$ which is a complete threefold $X$ with a two dimensional torus action. Its divisorial fan $\fan$ over $\PP^1$ looks like Figure~\ref{fig:cotan}, cf. \cite[8.5]{divfans}.

\begin{figure}[b]
\centering
\subfigure[$\fan_0$]{\cotangfannullv}
\subfigure[$\fan_\infty$]{\cotangfaninftyv}
\subfigure[$\fan_1$]{\cotangfanonev}
\caption{Divisorial fan of $\PP(\Omega_{\PP^2})$, cf.\ Example \ref{ex:cotan}.}
\label{fig:cotan}
\end{figure}

Note that all polyhedral divisors have complete locus. We want to compute $\Gamma(X,-K_X)$, and use $K_{\PP^1} = -2\{0\}$ as a representation of the canonical divisor on $\PP^1$. By Theorem \ref{sec:thm-canonical-divisor} we have that 
\[ -K_X = 2D_{(\elements{0},(0,0))} + 2D_{(\elements{0},(0,1))} \,.\]
Using Corollary \ref{sec:cor-cartier2weil}, we can construct $h$ explicitly. We have $h_t(\rho_i)=-2$ for $1\leq i \leq 6\,,$ providing the weight polytope $\Box_h$ 
in Figure~\ref{fig:wpolytope}.

The next list displays the induced divisor $h^*(u)$ on $\PP^1$ for every weight $u=(u_1,u_2)\in \Box_h$, where a triple $(a,b,c)$ corresponds to $D(a,b,c) = a\{0\}+b\{\infty\}+ c\{1\}$.
\[\hspace*{-1.8ex}\begin{array}{ccc} \begin{array}{ccc} (0,0) & \mapsto & (2,0,0) \\ (1,0) & \mapsto & (2,-1,0) \\ (2,0) & \mapsto & (2,-2,0) \\ (-1,0) & \mapsto & (2,0,-1) \\ (-2,0) & \mapsto & (2,0,-2) \\ (0,1) & \mapsto & (2,-1,0) \\ (0,2) & \mapsto & (2,-2,0) \end{array} & \begin{array}{ccc} (0,-1) & \mapsto & (1,0,0) \\ (0,-2) & \mapsto & (0,0,0) \\ (-1,1) & \mapsto & (2,0,-1) \\ (-2,1) & \mapsto & (2,0,-2) \\ (-2,2) & \mapsto & (2,0,-2) \\ (-1,2) & \mapsto & (2,-1,-1) \\ (-1,-1) & \mapsto & (1,0,-1) \end{array} & \begin{array}{ccc} (1,1) & \mapsto & (2,-2,0) \\ (1,-1) & \mapsto & (1,0,0) \\ (2,-1) & \mapsto & (1,-1,0) \\ (2,-2) & \mapsto & (0,0,0) \\ (1,-2) & \mapsto & (0,0,0)  \\ & & \\ & & \end{array} \end{array}\]
Summing up yields $\dimension \Gamma(X,-K_X) = 27$. Furthermore we compute \[\rho_X = 1 + 0 + 3 -2 = 2,\] which is of course a classical result.
\end{ex}

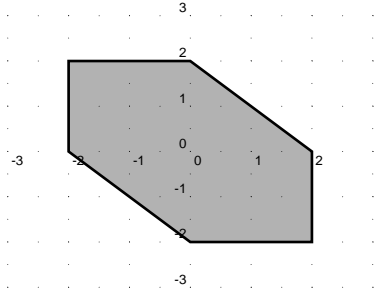
\begin{figure}[htbp]
\centering
\psset{xunit=.8cm,yunit=.6cm}
\begin{pspicture}(-3,-3)(3,3)%
%\psgrid[gridwidth=0.3pt,griddots=2,subgriddiv=1,gridlabels=5pt](0,0)(-3,-3)(3,3)
\newgray{gray1}{0.7}
\pspolygon[linewidth=1pt,fillstyle=solid,fillcolor=gray1](0,-2)(2,-2)(2,0)(0,2)(-2,2)(-2,0)
%\psline[linewidth=1pt]{-}(0,-2)(2,-2)
%\psline[linewidth=1pt]{-}(2,-2)(2,0)
%\psline[linewidth=1pt]{-}(0,2)(-2,2)
%\psline[linewidth=1pt]{-}(2,0)(0,2)
%\psline[linewidth=1pt]{-}(-2,2)(-2,2)
%\psline[linewidth=1pt]{-}(-2,2)(-2,0)
%\psline[linewidth=1pt]{-}(-2,0)(0,-2)
\psgrid[gridwidth=0.3pt,griddots=2,subgriddiv=1,gridlabels=5pt](0,0)(-3,-3)(3,3)
\end{pspicture}
\caption{The weight polytope $\Box_h$, cf.\ Example \ref{ex:cotan}.}
\label{fig:wpolytope}
\end{figure}

%%%%%%%%%%%%%%%%%%%%%%%%%%%%%%%%%%%%%%%%
\subsection{Positivity of line bundles}
\label{sec:ample_intersection}
%%%%%%%%%%%%%%%%%%%%%%%%%%%%%%%%%%%%%%%%
The goal of this section is to determine criteria for the ampleness of an invariant Cartier divisor and to give a method how to compute intersection numbers of semiample invariant Cartier divisors. We assume $\fan$ to be complete. Denote its tailfan by $\Sigma$.

\begin{defn}
For a cone $\sigma \in \Sigma(n)$ of maximal dimension in the tailfan and a point $P \in Y$ we get exactly one polyhedron $\Delta^\sigma_P \in \fan_P$ having tail $\sigma$. For a given support function $h = (h_P)_P$ we have
\[h_P|_{\Delta^\sigma_P} = \langle u^h(\sigma), \cdot \rangle + a^h_P(\sigma).\]
The constant part gives rise to a divisor on $Y$:
\[h|_{\sigma}(0) := \sum_P a^h_P(\sigma) P.\]
\end{defn}

\begin{thm}
\label{sec:thm-ample}
A $T$-Cartier divisor  $h \in \tcadiv(\fan)$ is semiample iff all $h_P$ are concave and $\deg h|_\sigma(0) < 0$ or some multiple of  $-h|_\sigma(0)$ is principal, i.e.\ $-h|_\sigma(0)$ is semiample.
\end{thm}

\begin{proof}
We first show that semiampleness follows from the above criteria. If $h$ is concave then so is $h_0$. This implies that the $u^h(\sigma)$ are exactly the vertices of $\Box_{h}$ and $h^*(u^h(\sigma))=h|_\sigma(0)$. The semiampleness for $h^*(u),\; u \in \Box_h \cap M$ now follows from the semiampleness at the vertices. Indeed, if $D, D'$ are semiample divisors on $Y$ then $D + \lambda (D'-D)$ with $0\leq \lambda \leq 1$ is also semiample. Observe that every vertex $(u,a_u)$ of the graph $\Graph_{h^*_P}$ corresponds to an affine piece of $h_P$. This again corresponds to a function  $f \chi^u$ with $\divisor(f) = a_u P$ on $U_P$ for some $\D \in \fan$ (see Section \ref{sec:construction-cartier-div}) with $P \in U_P \subset Y$ a sufficiently small neighborhood. Now $D_h|_{\pdv(\D|_{U_P})} = \divisor(f^{-1} \chi^{-u})$.
  
A point $(u,a_u)\in M \times \ZZ$ is a vertex of the graph $\Graph_{h^*}$ iff $(m u,m a_u)$ is a vertex of the graph $\Graph_{(m \cdot h)^*}$. Hence, after passing to a suitable multiple of $h$ we may assume that $h^*(u)$ is basepoint free with $f$ being a global section of $\CO(h^*(u))$. Then $f \chi^{u}$ is a global section of $\CO(D_h)$ generating $\CO(D_h)|_{\pdv(\D|_{U_P})}$.

For the other implication assume that there is a point $P \in Y$ such that $h_P$ is not concave. Then the same is true for all multiples $l h_P$. Thus, we can find an affine part $\langle u,\cdot \rangle -a_u$ of $l h_P$ such that $a_u > (l h_P)^*(u)$. But this implies that there is no global section $f\chi^u$ such that $\divisor(f) = a_u$ which contradicts the basepoint freeness of $D_{l h}$. Hence, $D_h$ cannot be semiample.
\end{proof}

\begin{cor}
A $T$-Cartier divisor $h \in \tcadiv(\fan)$ is ample iff all $h_P$ are strictly concave and for all tailcones $\sigma$ belonging to a polyhedral divisor $\D \in \fan$ with affine locus $\deg h|_{\sigma}(0) = \sum_P a^h_P(\sigma) < 0$, i.e.\ $-h|_{\sigma}(0)$ is ample.
\end{cor}

\begin{proof}
Note that for every invariant Cartier divisor $D_h$ the concaveness of $h$ implies that
$h_{\sigma}(0)$ is principal. Hence, the proof follows from Theorem \ref{sec:thm-ample} and the fact that $h_P$ is strictly concave if and only if for every support function $h'$ there is a $k \gg 0$ such that $h'+kh_p$ is concave.
\end{proof}

\begin{cor}
A $T$-Cartier divisor $h \in \tcadiv(\fan)$ is nef iff all $h_P$ are concave and $\deg h|_\sigma \leq 0$ for every maximal cone $\sigma \in \Sigma(n)$. 
\end{cor}

\begin{proof}
Using the equivariant Chow Lemma, we can pull back $D_h$ by an equivariant birational proper morphism $\phi: X(\fan') \to X(\fan)$. Hence, we may assume that $D_h$ lives on a projective $T$-variety $X':=X(\fan')$, i.e.\ there exists an ample divisor $D_{h'}$ on $X'$. It is easy to check that $D_h + \varepsilon D_{h'}$ is ample for $\varepsilon >0$ iff $h$ fulfills the above conditions. 
\end{proof}

Applying Proposition~\ref{prop:global_sections} to determine $\dim \Gamma(X,D_h)$, we are now able to compute intersection numbers. 

\begin{defn}
For a function $h^*:\Box\rightarrow \wdiv_\Q Y$ we define its {\em volume} to be
\[\vol h^* := \sum_P \int_{\Box} h^*_P \vol_{M_\RR}\,.\]
We associate a {\em mixed volume} to functions $h^*_1,\ldots, h^*_k$ by setting
\[V(h^*_1,\ldots, h^*_k):= \sum_{i=1}^k (-1)^{i-1} \sum_{1\leq j_1 \leq \ldots j_i \leq  k} \vol(h^*_{j_1} + \cdots + h^*_{j_i})\,.\]
\end{defn}

The following proposition is a special instance of \cite[Theorem 8]{timashevCDcomplOne}.

\begin{prop}\label{prop:intersection-numbers}
Let $\fan$ be a divisorial fan on a curve $Y$ with slices in $N \cong \ZZ^n$.
\begin{enumerate}
\item The self-intersection number of a semiample Cartier divisor $D_h$ is given by \label{item:prop-vol-vol}
\[(D_h)^{(n+1)} = (n+1)!\vol h^*.\]
\item Assume that $h_1, \ldots, h_{n+1}$ define semiample divisors $D_i$ on $X(\fan)$. Then 
\[(D_1 \cdots D_{n+1}) = (n+1)!V(h^*_1, \ldots, h^*_{n+1}).\]
\label{item:prop-vol-mixed}
\end{enumerate}
\end{prop}

\begin{proof}
If we apply (\ref{item:prop-vol-vol}) to every sum of divisors from $D_1,\ldots,D_{m+1}$ we get (\ref{item:prop-vol-mixed}) by the multilinearity and the symmetry of intersection numbers. To prove (\ref{item:prop-vol-vol}) we first recall that
\[(D_h)^{m+1} = \lim_{\nu \rightarrow \infty} \frac{(m+1)!}{\nu^{m+1}} \chi(X,\CO(\nu D_h)).\]
Invoking the equivariant Chow Lemma, we can assume $X:=\pdv(\fan)$ to be projective. So the higher cohomology groups are asymptotically irrelevant \cite[Thm. 6.7.]{MR1919457}. Hence,
\[(D_h)^{m+1} = \lim_{\nu \rightarrow \infty} \frac{(m+1)!}{\nu^{m+1}} h^0(X,\CO(\nu D_h)).\]
Note that $(\nu h)^*(u)=\nu \cdot h^*(\frac{1}{\nu}u)$. We can now bound $h^0$ by

\begin{equation} \label{eq:h-bound}\tag{a}
\begin{split}
\sum_{u \in \nu \Box_h \cap M} \left( \deg \lfloor \nu h^*\left({\textstyle \frac{1}{\nu}} u \right)\rfloor - g(Y) +1\right) & \leq  h^0(\CO(\nu D_h)) \\   & \leq  \sum_{u \in \nu \Box_h \cap M} \deg \lfloor \nu h^*\left({\textstyle \frac{1}{\nu}}  u\right) \rfloor + 1 .
\end{split}
\end{equation}

Furthermore, we have that

\begin{eqnarray*}
\lim_{\nu \rightarrow \infty} \frac{(m+1)!}{\nu^{m+1}} \sum_{u \in \nu \Box_h \cap M} \deg \lfloor \nu h^*\left({\textstyle \frac{1}{\nu}} u\right) \rfloor
& = & \lim_{\nu \rightarrow \infty} \frac{(m+1)!}{\nu^{m}} \sum_{u \in \Box_h \cap \frac{1}{\nu}M} \frac{1}{\nu} \deg \lfloor \nu h^*(u)\rfloor \\
&=& (m+1)! \int_{\Box_h} h^* \vol_{M_\RR}.
\end{eqnarray*}

Finally,
\[\lim_{\nu\rightarrow \infty} \frac{1}{\nu^{m+1}} \sum_{u \in \nu \Box_h \cap M} (g - 1) = (g-1)\lim_{\nu\rightarrow \infty} \frac{\# (\nu\cdot \Box_h \cap M)}{\nu^{m+1}}=0\,.\]
Passing to the limit in (\ref{eq:h-bound}), the term in the middle converges to $\vol h^*$. This completes the proof.
\end{proof}

\begin{ex}
Let $X = \PP_{\Omega_{\PP^2}}$ as in Example \ref{ex:cotan} and $D_h = -K_X$. An easy calculation shows that $(-K_X)^3 = 6 \cdot (18\frac{2}{3} -5\frac{1}{3} -5\frac{1}{3}) = 48$, matching the result already known from the classification of Fano threefolds.
\end{ex}

%%%%%%%%%%%%%%%%%%%%%%%%%%%%%%%%%%%%%%%%%%%%%%%%%%%%%%%%%%%%%%%%%%
\section{Comparing results in the case of affine $\C^*$-surfaces}
%%%%%%%%%%%%%%%%%%%%%%%%%%%%%%%%%%%%%%%%%%%%%%%%%%%%%%%%%%%%%%%%%%

Normal affine $\C^*$-surfaces are very well understood. Results concerning the divisor class group and the canonical divisor can be found in \cite[4.24, 4.25]{MR2020670} and references therein. We will shortly remind the reader of the notation used in \cite{MR2020670} where the Dolgachev-Pinkham-Demazure (DPD) construction is used for the explicit construction of (hyperbolic) affine $\C^*$-\,surfaces, and state the corresponding results.

%%%%%%%%%%%%%%%%%%%%%%%%%%%
\subsection{Elliptic Case}
%%%%%%%%%%%%%%%%%%%%%%%%%%%
Let $C$ be a smooth projective curve and $D = \sum_i \frac{e_i}{m_i} [a_i]$ a $\Q$-Cartier divisor with $\sum_i \frac{e_i}{m_i} > 0$ and $\gcd(e_i,m_i) = 1$. The cone construction provides an affine $\C^*$-surface $X = \spec A_{C,D}$ whose class group of divisors is given by
\[ \cl(X) = \frac{\pic C \oplus \bigoplus_i \Z [O_i]}{\langle \pi^*(a_i) = m_i [O_i], \, 0 = \sum_i e_i[O_i] \rangle}\,,\]
and the canonical divisor can be represented by
\[K_X = \pi^*K_C = \sum_i(m_i -1)[O_i]. \]
Here, $O_i = \pi^{-1}(a_i)$. The corresponding data in our language are $Y=C$, $\sigma = \Q_{\geq 0}$, and $\D = \sum_i [\frac{e_i}{m_i}, \infty) \otimes {a_i}$. Then $D = \D(1)$.

%%%%%%%%%%%%%%%%%%%%%%%%%%%%
\subsection{Parabolic Case}
%%%%%%%%%%%%%%%%%%%%%%%%%%%%
This time one considers a smooth affine curve $C$ together with a $\Q$-divisor $D = \sum_i \frac{e_i}{m_i}[a_i]$ with $\gcd(e_i,m_i) = 1$. The DPD construction yields an affine surface $X = \spec A_{C,D}$ whose class group of divisors has the following form 
\[\cl(X) = \frac{\pic C \oplus \Z[C] \oplus \bigoplus_{i=1}^{k} \Z[O_i]}{\langle \pi^*(a_i) = m_i[O_i], [C] = -\sum_{i=1}^k e_i[O_i]\rangle}\,.\] 
In addition one has that
\[K_X = \pi^* K_C + \sum_{i=1}^k (m_i-1)[O_i] - [C]\,.\]
Again, $O_i = \pi^{-1}(a_i)$. Using our notation gives us $Y=C$, $\sigma = \Q_{\geq 0}$, and $\D = \sum_i [\frac{e_i}{m_i}, \infty) \otimes {a_i}$. Once more $D = \D(1)$.

%%%%%%%%%%%%%%%%%%%%%%%%%%%%
\subsection{Hyperbolic Case}
%%%%%%%%%%%%%%%%%%%%%%%%%%%%
Let us consider a smooth affine curve $C$ and a pair $(D_+,D_-)$ 
\[D_+ = -\sum_i \frac{e_i}{m_i}a_i - \sum_j \frac{e_j^+}{m_j^+}b_j\,, \quad D_- = \sum_i \frac{e_i}{m_i}a_i + \sum_j \frac{e_j^-}{m_j^-}b_j \; \]
of $\Q$-divisors on $C$ such that $D_+ + D_- \leq 0$. Recall the convention that 
\[D_+(a_i) + D_-(a_i) = 0\,,\quad \textnormal{and} \quad D_+(b_j) + D_-(b_j) < 0\,.\]
Using this pair, the DPD construction provides us with an affine $\C^*$-surface
\[X = \spec A_{C,(D_+,D_-)}\,.\]
The class group of divisors $\cl(X)$ then is 
\[ \pic C  \oplus \sum_{i=1}^k \Z[O_i] \oplus \sum_{j=1}^l (\Z[\overline{O}_j^+] \oplus \Z[\overline{O}_j^-])\]
modulo the relations 
\[\begin{array}{rcl} \pi^*(a_i) & = & m_i[O_i] \\ \pi^*(b_j) & = & m_j^+[\overline{O}_j^+]-m_j^-[\overline{O}_j^-] \\ 0 & = & \sum_{i=1}^k e_i[O_i] + \sum_{j=1}^l (e_j^+[\overline{O}_j^+]-e_j^-[\overline{O}_j^-])\,. \end{array}\]
Furthermore,
\[K_X = \pi^* K_C + \sum_{i=1}^k (m_i-1) [O_i] + \sum_{j=1}^l \big((m_j^+ -1)[\overline{O}_j^+] + (-m_j^- -1)[\overline{O}_j^-] \big)\,. \]
As before, $O_i = \pi^{-1}(a_i)$, whereas $\pi^{-1}(b_j) = \overline{O}_j^+ \cup \overline{O}_j^-$. In our terms: $Y=C$, $\sigma = \{0\}$ and $\pD = \sum_i [v_i^-, v_i^+]\otimes{y_i}$ with $D_+ = \pD(1)$, and $D_- = \pD(-1)$.\\[2mm]

Using our formulae for the divisor class group and the canonical divisor yields the same results in all three cases, as can readily be seen by invoking Corollary \ref{cor:divclass} and Theorem \ref{sec:thm-canonical-divisor}. Hence, treating each case separately is no longer necessary.
%Another advantage is that these formulae can easily be applied to higher dimensional varieties with complexity one torus action. Furthermore, the language we use seems to yield a more natural means to study and treat the case of a complete $T$-variety, as to some extent has already been pointed out in \cite[8.1]{divfans}.

\bibliographystyle{alpha}
\bibliography{torusinvariant.bib}

\end{document}